\documentclass[12pt,reqno]{amsart}

\usepackage{amsfonts,color}

\def\Box{\vcenter{\vbox{\hrule\hbox{\vrule
     \vbox to 8.8pt{\hbox to 10pt{}\vfill}\vrule}\hrule}}}

\newcommand{\tr}{\text{Tr}}

\newcommand{\eproof}{\hfill$\Box$\vspace{4mm}}

\newcommand{\la}{\langle}
\newcommand{\ra}{\rangle}
\newcommand{\F}{{\mathbb F}}
\newcommand{\Q}{{\mathbb Q}}
\newtheorem{thm}{Theorem}

\newtheorem{lemma}[thm]{Lemma}

\newtheorem{remark}[thm]{Remark}

\numberwithin{equation}{section}

\begin{document}

\title{On homogeneous planar functions}
\author{Tao Feng}
\thanks{T. Feng is with Department of Mathematics, Zhejiang University, Hangzhou 310027, Zhejiang, China
(e-mail: tfeng@zju.edu.cn). The work of T. Feng was supported in part by the Fundamental Research Funds for the Central
Universities of China, Zhejiang Provincial Natural Science Foundation (LQ12A01019), National Natural Science
Foundation of China (11201418) and Research Fund for Doctoral Programs from the Ministry of Education of China (20120101120089).  }

\maketitle

\begin{abstract}
Let $p$ be an odd prime and $\F_q$ be the finite field with $q=p^n$ elements. A planar function $f:\F_q\rightarrow\F_q$ is called homogenous if $f(\lambda x)=\lambda^df(x)$ for all $\lambda\in\F_p$ and $x\in\F_q$, where $d$ is some fixed positive integer. We characterize $x^2$ as the unique homogenous planar function over $\F_{p^2}$ up to equivalence.
\end{abstract}

\section{Introduction}
Let $p$ be an odd prime, $n$ be a positive integer and $\F_q$ be the finite field with $q:=p^n$ elements. A function $f:\F_q\rightarrow\F_q$ is called {\it planar} if for any $a\in\F_q^*$, the function \[\Delta_af(x):=f(x+a)-f(x)-f(a)\] is a permutation polynomial.
Two planar functions $f_1$ and $f_2$ are {\it equivalent} if there are linearized permutation polynomials $L_1,\,L_2$, a linearized polynomial $L_3$ and a constant $c$ such that \[f_2(x)=L_2(f(L_1(x)))+L_3(x)+c.\] 
We call a planar function defined over $\F_q$ {\it homogeneous} if $f(\lambda x)=\lambda^d f(x)$ for all $\lambda\in\F_p$ and $x\in\F_q$, where $d$ is a fixed natural number. It is easy to see that $d\equiv2\pmod{p-1}$ according to the classification of planar functions over $\F_p$.
   
The study of planar functions originates from Dembowski and Ostrom's work \cite{DO} on affine planes of order $n$ with a large collineation group of order $n^2$; see also \cite{rs} for a short exposition. A planar function is homogeneous if and only if the corresponding affine plane has some extra automorphism group of order $p-1$.
There have been extensive studies in the literature on this topic since the work of Dembowski and Ostrom. A planar function is called DO type if it can be written in the form
\[
f(x)=\sum_{0\leq i\leq j\leq n-1}a_{ij}x^{p^i+p^j},\;a_{ij}\in\F_{q}.
\]
Such planar functions are closely related to semifields, and we refer the reader to \cite{bh,fs} for a most recent survey on this topic, where the reader can also find a list of all known planar functions and some references. For the combinatorial aspects of DO planar functions, we refer to \cite{WQWX,wz}.
 
Up to equivalence the only known planar functions not of DO type are due to Coulter-Matthews \cite{CM}, which are monomials in characteristic three. This disproves a conjecture of Dembowski and Ostrom which claims that all planar functions must be of DO type. Still, the conjecture remains open for any characteristic greater than three.
All the known planar functions are equivalent to one of DO type or Coulter-Matthews type, and planar functions of these two types are homogenous in our sense.  Most recently, the notion of planar functions have been defined and studied in the even characteristic in a series of papers, c.f. \cite{z1,z2,ky,yz,y_th}.

The planar functions over $\F_p$ have been classified in the monomial case by Johnson \cite{johnson} and in full by \cite{gluck,hiramine,rs} independently.  This contributes to the famous conjecture that all projective planes of prime order is desarguesian. It is natural to consider the classification of all planar functions over $\F_{p^2}$, but this turns out to be very difficult. The proof of \cite{hiramine} makes use of Hermite criterion for permutation polynomials, which has been followed by \cite{rc_ps,cl_p4} to classify all planar monomials over $\F_{p^2}$ and $\F_{p^4}$. Recently, all planar functions over $\F_q$  of the form $x^n$  with  $(n-1)^4 \le q$  have been classified in \cite{z3}.
It is the aim of this note to classify the homogeneous planar functions defined over $\F_{p^2}$. Our proof follows the same spirit of \cite{rs}.
\begin{thm}\label{thm_main}
Let $p$ be an odd prime and $f$ be a planar function defined over $\F_{p^2}$ such that $f$ is homogeneous, i.e., $f(\lambda x)=\lambda^df(x)$ for all $\lambda\in\F_p$ and $x\in\F_{p^2}$, where $d$ is some fixed positive integer. Then up to equivalence, we must have $f(x)=x^2$.
\end{thm}

\section{Preliminaries}
Throughout this note, we assume that $p$ is an odd prime and $f$ is a homogeneous planar functions defined over $\F_{p^2}$.  Let $T$ be a complete set of coset representatives of $\F_p^*$ in $\F_{p^2}^*$. Each element of $\F_{p^2}^*$ can be written uniquely as $xy$ with $x\in\F_{p^*}$ and $y\in T$. Also, take $\beta$ to be an element of $\F_{p^2}^*$ such that $\beta^{p-1}=-1$. Then clearly $\beta^2$ is a nonsquare in $\F_p^*$. We write $\tr_{\F_{p^2}/\F_p}$ as $\tr$ for short, and define
 \[
 Z_a:=\{z\in T|\tr(af(z))=0\},\;\forall\, a\ne 0.
 \]
We collect some basic observations in the following lemma.
\begin{lemma}\label{lem_1} With the above notations, the following holds.
\begin{enumerate}
\item For each $b\in \F_{p^2}^*$, there is a unique element $y_b\in T$ such that $\tr(by_b)=0$.
\item $f(0)=0$, $f(-x)=f(x)$ for any $x\in\F_{p^2}$, and $f$ is a $2$-to-$1$ map on $\F_{p^2}^*$. In particular, for distinct $y,y'\in T$, $f(y')f(y)^{-1}$ can not be a square of $\F_p^*$.
\item $d\equiv 2\pmod{p-1}$.
\end{enumerate}
We assume that $f(1)=1$ from now on.
\end{lemma}
\begin{proof}
(1) is trivial, and the first half of (2) is proved in \cite{KP,WQWX}. If $f(y')f(y)^{-1}=u^2$ for some $u\in\F_p^*$ and $y,y'\in T$, then $f(y')=f(uy)$ which implies $y'=\pm uy$ and thus $y'=y$. This proves the second half of (2). It follows that the multiset $\{f(y)\F_p^*|y\in T\}$ covers each coset of $\F_{p^2}^*/\F_p^*$ at most twice.

 Assume that (3) is false. Then according to the classification of (monomial) planar functions over $\F_p$ in \cite{johnson,gluck,hiramine,rs}, $x^d$ is not a planar function over $\F_{p}$, and so there exist two distinct elements $x,y\in\F_p$ such that $(x+1)^d-x^d=(y+1)^d-y^d$. It follows easily that $\Delta_1f(x)=\Delta_1f(y)$, which contradicts the fact that $f$ is planar.
\end{proof}

We set $\omega:=\exp(2\pi \sqrt{-1}/p)$. Let $\psi$ (resp. $\psi_2$) be the canonical additive character of $\F_p$ (resp. $\F_{p^2}$) defined by $\psi(x)=\omega^x$ (resp. $\psi_2(x)=\omega^{\tr(x)}$) for each $x\in\F_p$ (resp. $\F_{p^2}$).  Write $\delta_z$ for the Kronecker symbol which takes $0$ if $z\in\F_p^*$ and $1$ if $z=0$, and $ (\frac{\cdot}{p} )$ for the Legendre symbol modulo $p$.  We record here some well-known facts that we will use below.
\begin{lemma}\label{lem_sum}
\begin{enumerate}
\item   The Gauss sum $\sum_{x\in \F_p}\psi(x^2)=\sum_{x\in \F_p}\psi(x)(\frac{x}{p})=\sqrt{p^*}$,  where $p^*= (\frac{-1}{p} )p$.
\item $\sum_{i=0}^{p-1}(\frac{i}{p})=0$, and $\sum_{i=0}^{p-1}(\frac{i(i-l)}{p})=p\delta_l-1$ for each $l\in\F_p$.
\end{enumerate}
\end{lemma}
\begin{proof}
(1) is \cite[Theorem 1.2.4]{gjs}, and (2) is Eqn. (5.12) and Theorem 5.48 in \cite{ln}.
\end{proof}

For a fixed $a\ne 0$, we now define the following exponential sum
\begin{align*}
W(a,b):=\sum_{z\in\F_{p^2}}\psi_2(af(z)+bz)=\sum_{y\in T}\sum_{x\in\F_{p}}\psi(x^2\tr(af(y))+x\tr(by))-p.
\end{align*}
It is proved in \cite{pott,PZ} that there exist $\epsilon\in\{\pm 1\}$ and $0\le l\le p-1$ depending on $(a,b)$ such that $W(a,b)=\epsilon \omega^l p$.
 We look at the inner summation, and distinguish two cases:
\begin{enumerate}
\item $\tr(af(y))=0$. The inner sum is $p\delta_{\tr(by)}$.
\item $\tr(af(y))\ne0$. The inner sum is $\psi\left(\frac{(\tr(by))^2}{-4\tr(af(y))}\right)\left(\frac{\tr(af(y))}{p}\right)\sqrt{p^*}$ after simple calculations.
\end{enumerate}
\begin{lemma}\label{lem_x*}With the above notations, the following holds.
\begin{enumerate}
\item For each $a\ne 0$, $Z_a$ has size $0$ or $2$.
\item For each $x\ne 0$, there exist a unique $x^\ast$ (up to a $\pm$ sign) such that $f(x)=\beta^{2}f(x^\ast)$.
\end{enumerate}
\end{lemma}
\begin{proof}
 To prove (1), we use the above observation to compute that
\begin{align*}
W(a,0)=\sqrt{p^*}\sum_{y\in T\setminus Z_a } \left(\frac{\tr(af(y))}{p}\right)+p|Z_a|-p.
\end{align*}
This number falls in $\Q(\sqrt{p^*})$, and has modulus $p$, so must be $\pm p$; we refer the reader to \cite{ir} for such basic facts from algebraic number theory. Therefore, we get $\sum_{y\in T\setminus Z_a } \left(\frac{\tr(af(y))}{p}\right)=0$, and $|Z_a|-1\in\{1,-1\}$, i.e., $|Z_a|\in\{0,2\}$. This proves (1).

To prove (2), we take an element $a\ne 0$ such that  $\tr(af(x))=0$. Then $Z_a$ contains one element $T\cap x\F_p^*$ and thus another element $y\in T$. Now from $\tr(af(x))=\tr(af(y))=0$, we have that $f(y)f(x)^{-1}$ is in $\F_p^*$. It is a nonsquare in $\F_p^*$  by Lemma \ref{lem_1}. The claim now follows from the homogeneity of $f$ and the fact that $\beta^2$ is a nonsquare in $\F_p^*$.
\end{proof}

\begin{remark}\label{rem_val} Since $f(1)=1$, Lemma \ref{lem_x*} shows that there must be an element $u\in\F_{p^2}\setminus \F_p$ such that $f(u)=\beta^2$. The general linear group $GL(2,p)$ acts on the set of basis of $\F_{p^2}$ over $\F_p$ transitively, so there exist linearized permutation polynomials $L_1,\,L_2$ such that
\[
L_1(1)=1,\; L_1(\beta)=u;\quad L_2(1)=1,\;L_2(f(\gamma))=\gamma^2,
\]
where $\gamma$ is a primitive element of $\F_{p^2}$.
It follows that $g(x)=L_2(f(L_1(x)))$ is a homogeneous planar function equivalent to $f(x)$, and $g(1)=1$, $g(\beta)=\beta^2$ and $g(\gamma)=\gamma^2$. Therefore, up to equivalence, {\bf we assume  that
\begin{align}\label{eqn_val}
f(1)=1,\; f(\beta)=\beta^2,\; f(\gamma)=\gamma^2.
\end{align}
}\noindent Using  Lemma \ref{lem_x*}, we conducted a computer search which  took seconds to check that a homogeneous planar function with these three prescribed values is exactly $x^2$ when $p=3,5,7$. Therefore, {\bf we assume that $p\ge 11$ below}.
\end{remark}

\begin{lemma}\label{lem_zero}
For $x\ne 0$, $\tr(\beta f(x))=0$ if and only if $x\in\F_p^*\cup\beta\F_p^*$.
\end{lemma}
\begin{proof}
For $x\ne 0$, $\tr(\beta f(x))=0$ if and only if $f(x)\in\F_p^*$. The claim now follows from  Eqn. \eqref{eqn_val} and Lemma \ref{lem_1}.
\end{proof}

Take $a,b$ to be two nonzero elements in $\F_{p^2}$.  Recall that we use $y_b$ for the unique element of $T$ such that $\tr(by_b)=0$. Also, we define $n_i$ to be the size of $\left\{y\in T\setminus Z_a:\,\frac{(\tr(by))^2}{-4\tr(af(y))}=i\right\}$ for $i\in\F_p^*$, and $n_0$ for the size of $\left\{y\in T\setminus Z_a:\,\tr(by)=0\right\}$.  It is clear that $n_0=0$ if $y_b\in Z_a$ and $n_0=1$ otherwise. We have $\sum_{i\in\F_p}n_i=p+1-|Z_a|$, and $\sum_{z\in Z_a}\delta_{\tr(bz)}=|\{z\in Z_a:\,\tr(bz)=0\}|=1-n_0$.

Similar to the $b=0$ case, we compute that
\begin{align*}
\epsilon \omega^l p&=W(a,b)=\sqrt{p^*}\sum_{y\in T\setminus Z_a}\psi\left(\frac{(\tr(by))^2}{-4\tr(af(y))}\right)
\left(\frac{\tr(af(y))}{p}\right)+p\sum_{y\in Z_a}\delta_{\tr(by)}-p\\
&=\sqrt{p^*}\left[\sum_{i\in\F_p^*}n_i\left(\frac{-i}{p}\right)\omega^i+ n_0\left(\frac{\tr(af(y_b))}{p}\right)\right]-pn_0.
\end{align*}
Dividing  both sides by $(\frac{-1}{p})\sqrt{p^*}$ and using the fact $\sqrt{p^*}=\sum_{i\in\F_p}(\frac{i}{p})\omega^i$, we get
\begin{align*}
\sum_{i\in\F_p^*}n_i\left(\frac{i}{p}\right)\omega^i&+ n_0\left(\frac{-\tr(af(y_b))}{p}\right)
=(n_0+\epsilon\omega^l)\cdot\sum_{i\in\F_p}\left(\frac{i}{p}\right)\omega^i
=\sum_{i\in\F_p}\left[n_0\left(\frac{i}{p}\right)+\epsilon\left(\frac{i-l}{p}\right)\right]\omega^i.
\end{align*}
Since the minimal polynomial of $\omega$ over $\mathbb{Q}$ is $1+x+\cdots+x^{p-1}$, there exists a constant $c$ such that
\begin{align}\label{eqn_10}
&n_i\left(\frac{i}{p}\right)=n_0\left(\frac{i}{p}\right)+\epsilon\left(\frac{i-l}{p}\right)+c,
\quad\forall\, i\ne 0\\\label{eqn_11}
&n_0\left(\frac{-\tr(af(y_b))}{p}\right)=n_0\epsilon\left(\frac{-l}{p}\right)+c.
\end{align}
which gives $n_i =n_0 +\epsilon\left(\frac{i(i-l)}{p}\right)+c\left(\frac{i}{p}\right)$ for all $i\ne 0$. Recall that $\sum_{i\in\F_p}n_i=p+1-|Z_a|$. Taking the summation of $n_i$'s over $i\in\F_p^*$,  we  have
\begin{equation}\label{eqn_1}
p+1-|Z_a|-n_0=\sum_{i\in\F_p^*}n_i=(p-1)n_0+\epsilon(p\delta_l-1),
\end{equation}

\begin{lemma}\label{lem_c}
We have $n_0=1-\delta_l$, $c=0$ and $\epsilon=|Z_a|-1$. In particular,
\begin{enumerate}
\item If $n_0=0$, then $l=0$, $|Z_a|=2$, and $n_i=1$ for each $i\ne 0$.
\item If $n_0=1$, then $l\ne 0$, and $n_i =1 +\epsilon\left(\frac{i(i-l)}{p}\right)\le 2$ for each $i\ne 0$.
\end{enumerate}
\end{lemma}
\begin{proof}
By taking modulo $2$ of \eqref{eqn_1}, we have $ |Z_a|+n_0\equiv \delta_l+1\pmod{2}$, so $n_0\equiv 1-\delta_l\pmod{2}$. Since $0\le n_0\le 1$, it follows that $n_0=1-\delta_l$.

By adding up the  $p-1$ equations of \eqref{eqn_10} and using Lemma \ref{lem_sum}, we get that $\sum_{i\in\F_p^*}n_i\left(\frac{i}{p}\right)=(p-1)c-\epsilon(\frac{-l}{p})$. Therefore, $(p-1)|c|\leq \sum_{i\in\F_p^*}n_i+1=p+2-|Z_a|-n_0<2(p-1)$, and $c$ must be in the set $\{0,1,-1\}$.
If $n_0=0$, then from \eqref{eqn_11} we get $c=0$. If $n_0=1$, then $l\ne 0$, $y_b\in T\setminus Z_a$, and so $\tr(af(y_b))\ne 0$.  Eqn. \eqref{eqn_11} gives that $c$ is even, so $c=0$.

We now have shown that
$n_i =1-\delta_l +\epsilon\left(\frac{i(i-l)}{p}\right)$ for all $i\in \F_p$. Their sum is equal to $p(1-\delta_l)+\epsilon(p\delta_l-1)$ by Lemma \ref{lem_sum}, which should be equal to $p+1-|Z_a|$. This yields that $p(-1+\epsilon)\delta_l=\epsilon+1-|Z_a|$. Since $-2\le \epsilon+1-|Z_a|\le 2$, we must have $\epsilon+1-|Z_a|=0$ and $(-1+\epsilon)\delta_l=0$, i.e., $\epsilon=|Z_a|-1$ and $(\epsilon-1)\beta_l=0$. This completes the proof.
\end{proof}

\section{The ovals}

For each $t\ne 0$, we define the following subset of $PG(2,p)$:
\[
S_t:=\left\{P(x)=\la \tr(tx)^2,\tr(f(x)),\tr(\beta f(x))\ra|\,x\in \F_{p^2}^*\right\}.
\]
We observe that
\begin{equation}\label{eqn_f}
f(x)=\frac{1}{2}\tr(f(x))+\frac{1}{2\beta}\tr(\beta f(x)).
\end{equation}
In particular, no element of $S_t$ is zero. Also, $S_t=\{P(y)|\,y\in T\}$ using the homogeneity of $f$.  The set $S_t$ has size $p+1$: if $y,z\in T$ are distinct and $P(y)=P(z)$, then there exists a constant $\lambda\in\F_p^*$ such that
\[
\tr(ty)^2=\lambda\tr(tz)^2,\quad f(y)=\lambda f(z).
\]
The latter condition implies that $\lambda$ is a nonsquare by Lemma \ref{lem_1}. Then the first holds only if $\tr(ty)=0$ and $\tr(tz)=0$, which implies that $y=z$: a contradiction.

\begin{thm}
For each $t\ne 0$, $S_t$ is an oval.
\end{thm}
\begin{proof}
A typical projective line has the form $[u,v,w]:=\{\la x,y,z\ra|\,ux+vy+wz=0,\,x,y,z\in\F_p\}$, where $u,v,w\in\F_p$  are not all zeros. The number of its intersection points with the set $S_t$  is equal to the number of solutions to
\begin{align}\label{eqn_oval}
u\tr(ty)^2+\tr((v+w\beta) f(y))=0,\quad y\in T.
\end{align}
We need to show that it has at most two solutions, so that $S_t$ is an oval by definition.

Write $a':=v+w\beta$. The elements $a',u$ are not both zero. If $a'\ne 0$, $u=0$, then Eqn. \eqref{eqn_oval} reduces to $\tr(a'f(y))=0$, which has $|Z_{a'}|\in\{0,2\}$ solutions in $T$. If $a'=0$, $u\ne 0$, the equation reduces to $\tr(ty)=0$, which has a unique solution in $y_t\in T$ by Lemma \ref{lem_1}. So we assume $a'u\ne 0$ from now on.

Now define $a:=a'u^{-1}$, $b:=2t$, so that Eqn. \eqref{eqn_oval} has the form $\tr(by)^2=-4\tr(af(y))$. We have $y_b=y_t$, since $b\F_p^*=t\F_p^*$. We need to consider two cases.
\begin{enumerate}
\item In the  case $\tr(af(y_b))= 0$, i.e., $y_b\in Z_a$,  $y_b$ is clearly a solution of the equation, while the other element of $Z_a$ is not. For those $y\in T\setminus Z_a$, there are $n_1=1$ of them that are solutions to Eqn. \eqref{eqn_oval}  by Lemma \ref{lem_c}. In total, there are exactly two solutions.
\item In the  case $\tr(af(y_b))\ne 0$, i.e., $y_b\not\in Z_a$, no element of $Z_a$ is a solution to Eqn. \eqref{eqn_oval}.  For those $y\in T\setminus Z_a$, there are $n_1=1 +\epsilon\left(\frac{i(i-l)}{p}\right)\le 2$ of them that are solutions to Eqn. \eqref{eqn_oval} by Lemma \ref{lem_c}. In total, there are at most two solutions in this case.
\end{enumerate}
To sum up, we have shown that each projective line intersects $S_t$ in at most two points. This proves that $S_t$ is an oval.
\end{proof}

Now that $S_t$ is an oval, it must be a conic by the well-known result of Segre \cite{HP,segre}. Therefore, it satisfies a quadratic equation
\[
Q(X_0,X_1,X_2):=\sum_{0\leq i\leq j\leq 2}c_{ij}X_iX_j
\]
for some constants $c_{i,j}\in\F_p$ which only depend on $t$. By the previous argument, the line $[1,0,0]$ intersects $S_t$ at the unique point $P(y_t)=\la 0,\tr(f(y_t)),\tr(\beta f(y_t))\ra$,  so is the tangent line at $P(y_t)$. Therefore,
\[
\left[\frac{\partial Q}{\partial X_0},\frac{\partial Q}{\partial X_1},\frac{\partial Q}{\partial X_2}\right]_{P(y_t)}=\left[1,0,0\right].
\]
This yields that
\[
\tr\left(f(y_t)(2c_{11}+c_{12}\beta)\right)=0,\quad \tr\left(f(y_t)(c_{12}+2c_{22}\beta)\right)=0.
\]
Hence,  $2c_{11}+c_{12}\beta$ and $c_{12}+2c_{22}\beta$ must be linearly dependent over $\F_p$, i.e., $c_{12}^2=4c_{11}c_{22}$. Therefore, there exist constants $\lambda,h_0,h_1\in\F_p$ such that
\[
c_{11}X_1^2+c_{12}X_1X_2+c_{22}X_2^2=\lambda(h_1X_1+h_2X_2)^2.
\]
Observe that $\la 1,0,0\ra\not\in S_t$, so $c_{00}\ne 0$ and we set $c_{00}=1$ below. Now $Q$ takes the following form
\[
Q(X_0,X_1,X_2)=X_0^2+X_0(c_{01}X_1+c_{02}X_2)+\lambda(h_1X_1+h_2X_2)^2.
\]
Since $Q$ is nondegenerate, we have $\lambda \ne 0$ and the two vectors $(c_{01},c_{02})$ and $(h_1,h_2)$ are linearly independent over $\F_p$, i.e., $h_1c_{02}-h_2c_{01}\ne 0$.

\begin{lemma}\label{lem_c01} We have $c_{01}=-\tr(t^2)$, and
we may set $\lambda=\beta^{-2}$ and $h_1=\frac{1}{2}\tr(t)\tr(\beta t)$ for all $t\ne 0$ by rescaling $\lambda,\,h_1,\,h_2$ properly.
\end{lemma}
\begin{proof}
We look at the two points $P(1)=\la\frac{1}{2}\tr(t)^2,1,0\ra$ and $P(\beta)=\la \frac{1}{2}(\beta^{-1}\tr(\beta t))^2,1,0\ra$. Plugging them into the expression of $Q$, we see that the equation $X_0^2+c_{01}X_0+\lambda h_1^2=0$ has the two solutions  $\frac{1}{2}(t+t^p)^2$ and $\frac{1}{2}(\beta^{-1}\tr(\beta t))^2=\frac{1}{2}(t-t^p)^2$. It follows that
\[
c_{01}=-\frac{1}{2}(t+t^p)^2-\frac{1}{2}(t-t^p)^2=-t^2-t^{2p},\quad
\lambda h_1^2=\frac{1}{4}\left(t^2-t^{2p}\right)^2.
\]
If $t^2\in\F_p^*$, then it follows that $h_1=0=\frac{1}{2}\tr(t)\tr(\beta t)$. We leave the proof about the claim that $\lambda$ is a nonsquare in this case until we have derived Eqn. \eqref{eqn_sq} below. Assume that $t^2\not\in\F_p^*$. If $\lambda=w^2$ with $w\in\F_p^*$, then  $\pm 2wh_1=t^2-t^{2p}$ which gives a contradiction upon taking trace. Hence $\lambda$ is a nonsquare of $\F_p^*$.  We thus set $\lambda=\beta^{-2}$ below by properly rescaling $h_1$ and $h_2$. It follows that
\begin{equation*}\label{eqn_h1}
h_1=\frac{\beta}{2}(t^2-t^{2p})=\frac{1}{2}\tr(t)\tr(\beta t)
\end{equation*}
upon multiplying both $h_1$ and $h_2$ with $\pm 1$ simultaneously.
\end{proof}

\begin{lemma}\label{lem_int}$\la 1,0,0\ra$ is an internal point, i.e., each line through it has either $0$ or $2$ points of $S_t$. The line through $\la 1,0,0\ra$ and $P(x)$ also passes through the point $P(x^*)$, with $x^*$ as defined in Lemma \ref{lem_x*}.
\end{lemma}
\begin{proof}Let $[ 0,u,v]$ be a line in $PG(2,p)$ through $\la 1,0,0\ra$. A point $P(x)\in S_t$ lies on it if  $\tr((u+v\beta)f(x))=0$. For each $x\in\F_{q^2}^*$, there exists $x^\ast\in\F_{p^2}^*$ such that $f(x)=\beta^2 f(x^\ast)$ by Lemma \ref{lem_x*}. Therefore, $P(x)=\la \tr(tx)^2,\tr(f(x)),\tr(\beta f(x))\ra$, and
\begin{equation*}\label{eqn_Px*}
P(x^\ast)=\la\beta^{2}\tr(tx^\ast)^2,\tr(f(x)),\tr(\beta f(x))\ra.
\end{equation*}
It is clear that $[ 0,u,v]$ passes through $P(x)$ if and only if it passes through the other point $P(x^\ast)$.
\end{proof}

Let us work out Lemma \ref{lem_int} in the coordinate form in details. Take $X=\la X_0,X_1,X_2\ra$ to be a point of $S_t$. Let $X^\ast=\la X_0+\lambda,X_1,X_2 \ra$ be the other point of $S_t$ lying on the line through $\la 1,0,0\ra$ and $X$. Then it is straightforward to compute that
\[
Q(X^\ast)=\lambda(\lambda+2X_0+c_{01}X_1+c_{02}X_2)=0,
\]
which yields that $\lambda=-2X_0-c_{01}X_1-c_{02}X_2$, and so $X^\ast=\la-X_0-c_{01}X_1-c_{02}X_2,X_1,X_2\ra$. We have thus shown that the central collineation determined by the following matrix
\[
M=\begin{pmatrix}-1&0&0\\-c_{01}&1&0\\-c_{02}&0&1\end{pmatrix}
\]
interchanges the points of $S_t$ lying on the same secant through $\la 1,0,0\ra$.\\

\section{Proof of the main result}
We are now ready to present the proof of our main result, which follows from a series of lemmas. By Lemma \ref{lem_int}, $P(x)$ and $P(x^\ast)$ lies on the same secant through $\la 1,0,0\ra$.   For the ease of notations, we shall write
\[
f_1:=\tr(f(x)),\quad f_\beta:=\tr(\beta f(x)).
\]

\begin{lemma}There exist constants $A,B\in\F_p$ such that $c_{02}=\tr(At^2)+Bt^{p+1}$, and
\begin{align}\label{eqn_AB}
  Af_\beta+\beta^2(x^\ast)^2+x^2-f_1=0,\quad Bf_\beta+2\beta^2(x^\ast)^{1+p}+2x^{1+p}=0
\end{align}
hold for all $x\in\F_{p^2}^*$.
\end{lemma}
\begin{proof}
From the connection between $P(x)$ and $P(x^\ast)$ we have derived after Lemma \ref{lem_int}, it follows that
\begin{equation}\label{eqn_lhs}
\beta^2\tr(tx^\ast)^2=-\tr(tx)^2-c_{01}f_1-c_{02}f_\beta.
\end{equation}
for all $x\in\F_{p^2}^*$. Recall that $c_{01}=-\tr(t^2)$ by Lemma \ref{lem_c01}. Expanding this equation, we get
\begin{align}\label{eqn_c02}
-c_{02}&f_\beta=\beta^2\tr(tx^\ast)^2+\tr(tx)^2+c_{01}f_1\notag\\
&=\tr\left((\beta^2(x^\ast)^2+x^2-f_1)t^2\right)
+2\left[\beta^2(x^\ast)^{1+p}+x^{1+p}\right]t^{p+1}.
\end{align}
By specializing to any  $x\ne0$ such that $f_\beta\ne 0$, we see that $c_{02}$, as a polynomial in the variable $t$, involves only the monomials $t^2$, $t^{2p}$ and $t^{p+1}$. In other words,  there are constants $A,B$ independent of $t$ such that
$c_{02}=\tr(At^2)+Bt^{p+1}$.

After plugging the expression of $c_{02}$ into Eqn. \eqref{eqn_c02}, we see that it has degree at most $2p$ in $t$. However, it holds for all $t\in\F_{p^2}^*$. Since $p^2-1>2p$ , it follows that this is a zero polynomial in $t$. Comparing the coefficients of $t^2$, $t^{2p}$ and $t^{p+1}$, we get the second half of the claim.
\end{proof}

\begin{remark}All elements of $S_t$ satisfy the equation $Q(X_0,X_1,X_2)=0$, so
\begin{equation}\label{eqn_ovalp}
 \tr(tx)^2\left(\tr(tx)^2+c_{01}f_1+c_{02}f_2\right)=-\lambda\left(h_1f_1+h_2f_\beta\right)^2
\end{equation}
holds for all $x\in\F_{p^2}^*$. By using Eqn. \eqref{eqn_lhs}, this reduces to
\begin{equation}\label{eqn_sq}
\beta^2\tr(tx)^2\tr(tx^\ast)^2=\lambda(h_1f_1+h_2f_\beta)^2.
\end{equation}
By choosing $x$ such that $\tr(tx)\tr(tx^\ast)\ne 0$, we see that $\lambda$ is a nonsquare in $\F_p^*$ for all $t\ne 0$. This completes the proof of Lemma \ref{lem_c01}.
\end{remark}

\begin{lemma} There exist constants $C,D\in\F_{p^2}$, $E\in\F_p$ and a sign function $\eta(t)\in\{\pm 1\}$ such that
\[
h_2=\tr\left((\eta(t)C-D)t^2\right) +\eta(t)Et^{p+1}.
\]
holds for all $t\in\F_{p^2}^*$.
\end{lemma}
\begin{proof} Recall that $\lambda=\beta^{-2}$ by Lemma \ref{lem_c01}.
Eqn. \eqref{eqn_sq} yields that
\begin{align*}
h_1f_1+h_2f_\beta=\epsilon \beta^2\tr(tx)\tr(tx^\ast)
\end{align*}
for some $\epsilon=\pm 1$. The sign $\epsilon$ here may depend on $t$ and $x$. It comes very natural since we have the freedom of $\pm$ in choosing $x^*$.  Using the expression of $h_1$ in Lemma \ref{lem_c01}, we get
\begin{align}
h_2f_\beta&=\epsilon \beta^2\tr(tx)\tr(tx^\ast)-\frac{1}{2}\tr(t)\tr(\beta t)f_1\notag\\\label{eqn_RHS}
&=\tr\left((\epsilon \beta^2xx^\ast-\frac{\beta}{2}f_1)t^2\right)
 +\epsilon \beta^2\left(x(x^\ast)^p+x^px^\ast\right)t^{p+1}.
\end{align}
If we take a specific $x\ne 0$ such that $f_\beta\ne 0$, then we see that there are constants $C,D\in\F_{p^2}$, $E\in\F_p$, and a sign function $\eta(t)\in\{\pm 1\}$ such that $h_2$ takes the form as stated in the Lemma.
\end{proof}
Plugging this expression of $h_2$ into Eqn. \eqref{eqn_RHS}, we compute that
\begin{align}
 \tr\left((\epsilon \beta^2xx^\ast-\frac{\beta}{2}f_1 -\eta(t)Cf_\beta+Df_\beta)t^2\right)
 +\left(\epsilon \beta^2\left(x(x^\ast)^p+x^px^\ast\right)  - \eta(t)f_\beta E\right)t^{p+1}=0.\label{eqn_h2}
\end{align}

\noindent {\it For the remaining arguments, recall that we assume $p>7$ by Remark \ref{rem_val}.}
\begin{lemma}\label{lem_etat} We may set $\eta(t)$ to be a constant $\nu\in\{1,-1\}$ for all $t\ne 0$.
\end{lemma}
\begin{proof}
We observe that $c_{01}$, $c_{02}$ and $h_1$ involve the monomials $t^2,t^{2p},t^{p+1}$ as polynomials in $t$, and $h_2$ involves the same monomials plus the sign function $\eta(t)$.
The left hand side of Eqn. \eqref{eqn_ovalp} is equal to
\begin{align*}
H(t):=&\tr(tx)^2\left[\tr\left((-f_1+Af_\beta+x^2)t^2\right)+(2x^{p+1}+Bf_\beta) t^{p+1}\right]
\end{align*}
The right hand side, $-\beta^{-2}(h_1f_1+h_2f_\beta)^2$, has the form $F(t)+\eta(t)G(t)$ for some polynomials $F(t)$ and $G(t)$ by direct expansion. The three polynomials $F,G,H$ each has degree at most $4p$. Now Eqn. \eqref{eqn_ovalp} has the form $H(t)=F(t)+G(t)\eta(t)$, which holds for all $x\ne 0$ and $t\ne 0$.  For some $\nu\in\{1,-1\}$, there are at least $\frac{p^2-1}{2}$ of $t$'s such that $\eta(t)=\nu$. When $p>7$, we have $\frac{p^2-1}{2}>4p$, and  $H(t)=F(t)+\nu G(t)$ has to be a zero polynomial in $t$. In particular, this means that Eqn. \eqref{eqn_ovalp}  holds for all $x\ne 0$ and $t\ne 0$ when we take $\eta(t)$ to be a constant $\nu$. This completes the proof.
\end{proof}

With $\eta(t)$ as a constant $\nu\in\{1,-1\}$ and the same argument as in the proof of Lemma \ref{lem_etat},  we see from Eqn. \eqref{eqn_h2} that $\epsilon$ is independent of $t$ and depends on $x$ only. Moreover, as a polynomial in $t$, Eqn. \eqref{eqn_h2} has degree at most $4p<p^2-1$ and is satisfied by all $t\ne 0$, so is a zero polynomial. By comparing coefficients, we obtain the following relations:
\begin{align*}
\epsilon \beta^2xx^\ast=\frac{\beta}{2}f_1+\nu Cf_\beta-Df_\beta ,\quad
\epsilon \beta^2\left(x(x^\ast)^p+x^px^\ast\right)  = \nu Ef_\beta .
\end{align*}
Together with Eqn. \eqref{eqn_AB} which we repeat here:
\[
Af_\beta+\beta^2(x^\ast)^2+x^2-f_1=0,\quad Bf_\beta +2\beta^2(x^\ast)^{1+p}+2x^{1+p}=0,
\]
we get $f_1=Af_\beta+\beta^2(x^\ast)^2+x^2$, and plug it into the others to get
\begin{align}\label{eqn_fb}
(x-\epsilon \beta x^\ast)^2&=-2\beta^{-1}(\beta A/2+\nu C-D)f_\beta
\end{align}
and
\begin{align}\label{eqn_last}
-x(\epsilon\beta x^\ast)^p+x^p(\epsilon\beta x^\ast)= \nu \beta^{-1}E f_\beta ,\quad
x^{1+p}-(\epsilon \beta x^*)^{p+1}=-\frac{B}{2}f_\beta.
\end{align}
\begin{lemma}\label{lem_C1} We have
$C_1:=-2\beta^{-1}(\beta A/2+\nu C+D)=0$.
\end{lemma}
\begin{proof}  Assume that $C_1\ne 0$. Write $\hat{x}:=\epsilon \beta x^\ast$.  Assume that $x\not\in \F_p^*\cup\beta \F_p^*$, so that $f_\beta\ne 0$ by Lemma \ref{lem_zero}.  Then $x\ne \hat{x}$ by Eqn. \eqref{eqn_fb}.  Using \eqref{eqn_last} we compute that
\[
(x-\hat{x})(x+\hat{x})^p=(-B/2-\nu\beta^{-1}E)f_\beta.
\]
Together with Eqn. \eqref{eqn_fb}, this gives  $(x+\hat{x})^p=C_2\cdot (x-\hat{x})$ for some constant $C_2$.

Since $f_1=Af_\beta+\hat{x}^2+x^2$ and $f_1=f_1^p$, we have
\begin{align*}
(A-A^p)f_\beta&=x^{2p}+\hat{x}^{2p}-x^2-\hat{x}^2\\
&=(x^p+\hat{x}^p)^2-2(x\hat{x})^p-2x\hat{x}-(x-\hat{x})^2\\
&=(C_2^2-1)(x-\hat{x})^2-2\tr(x\hat{x}).
\end{align*}
Together with the fact that $(x-\hat{x})^2=C_1 f_\beta$, we see that for such $x$ we have $\tr(x\hat{x})=C_3f_\beta$ for some constant $C_3$ which clearly lies in $\F_p$.
We now compute that
\[
2f_1=\tr(f_1)=\tr(A)f_\beta+\tr((x-\hat{x})^2)+2\tr(x\hat{x})=(\tr(A+C_1)+2C_3)
    \cdot f_\beta,
\]
i.e., $f_1 =C_4f_\beta $ for some constant $C_4\in\F_p$. This means that $\tr((1-\beta C_4)f(x))=0$ for all $x\not\in\F_p^*\cup\beta\F_p^*$. It follows that $1-\beta C_4=0$, which is a contradiction.
\end{proof}

\noindent{\bf Proof of Theorem \ref{thm_main}.} The case $p=3,5,7$ is confirmed by a computer search as commented after Lemma \ref{lem_x*}, so we assume that $p>7$. Since $C_1=0$ by Lemma \ref{lem_C1}, Eqn. \eqref{eqn_fb} gives that $x=\hat{x}$ for all $x\ne 0$. It follows that $f_1 =Af_\beta +2x^2$, and so $A$ clearly is not in $\F_p$. From $f_1=f_1^p$ we have
$(A^p-A)f_\beta=2(x^2-x^{2p})=2\beta^{-1}\tr(\beta x^2)$,
which shows that $f_\beta=C_5\cdot\tr(\beta x^2)$ for some constant $C_5\in\F_p^*$. Then
$2f_1=\tr(f_1)=\tr(A)f_\beta+2\tr(x^2)=\tr(C_6 x^2)$
for some constants $C_6\in\F_{p^2}$. Finally, $f=\frac{1}{2}f_1+\frac{1}{2\beta}f_\beta=C_7x^2+C_8x^{2p}$
for some constant $C_7,C_8\in\F_{p^2}$. It now follows from $f(1)=1$ and $f(\gamma)=\gamma^2$ in Eqn. \eqref{eqn_val} that $C_7+C_8=1$, $C_7\gamma^2+C_8\gamma^{2p}=\gamma^2$. This gives $C_7=1$ and $C_8=0$, i.e., $f(x)=x^2$. This completes our proof.\eproof

\noindent{\bf Acknowledgements.} To be added. 

\end{document}